\documentclass[12pt]{article}

\usepackage[a4paper, total={6.4in, 8.5in}]{geometry}
\usepackage{amsmath,amssymb,latexsym}
\usepackage{stmaryrd}
\usepackage{theorem}
\usepackage{eulervm}
\usepackage{color}
\usepackage{hyperref}
\hypersetup{
    colorlinks = true,
    linkcolor = blue,
    anchorcolor = red,
   citecolor = blue,
    filecolor = red,
    pagecolor = red,
    urlcolor = blue}

\newcommand{\N}{\mathcal{S}} 
\newcommand{\M}{{\mathcal{M}_{g,1}^{\N}}} 
\newcommand{\C}{\mathcal{C}} 
\newcommand{\LO}{\mathcal{O}} 
\renewcommand{\k}{\mathbf{k}} 
\newcommand{\proj}{\mathbb{P}} 
\newcommand{\Nat}{\mathbb{N}} 
\renewcommand{\l}{\ell} 
\newcommand{\F}[1]{F_{#1}} 
\newcommand{\AF}{\mathbb{A}} 

\DeclareMathOperator*{\Spec}{Spec} %
\DeclareMathOperator*{\Proj}{Proj} %
\DeclareMathOperator*{\ewt}{ewt}

\DeclareMathOperator*{\End}{End}

\theoremstyle{change}
\theorembodyfont{\rmfamily}
\newtheorem{thm}{Theorem}[section]
\newtheorem{lem}[thm]{Lemma}
\newtheorem{cor}[thm]{Corollary}
\newtheorem{prop}[thm]{Proposition}
\newtheorem{rmk}[thm]{Remark}
\newtheorem{syzlem}[thm]{Syzygy Lemma}

\newenvironment{proof}{\paragraph{Proof}}{\hfill$\square$\par\bigskip}
\newcommand\extrafootertext[1]{%
    \bgroup
    \renewcommand\thefootnote{\fnsymbol{footnote}}%
    \renewcommand\thempfootnote{\fnsymbol{mpfootnote}}%
    \footnotetext[0]{#1}%
    \egroup
}

\def\Y(#1){Y_{#1}}
\def\f(#1)(#2){f_{#1}^{(#2)}}
\def\g(#1)(#2){g_{#1}^{(#2)}}
\def\ff(#1)(#2){f_{#1,#2}}
\def\gg(#1)(#2){g_{#1,#2}}

\long \def\comment#1\endcomment{}

\begin{document}

\date{}
\title{On nonnegatively graded Weierstrass points}

\author{Andr\'e Contiero%
\thanks{The first author is partially supported by Funda\c c\~ao de Amparo \`a Pesquisa do Estado
de Minas Gerais (FAPEMIG) grant no. APQ-00798-18.}, Aislan Leal Fontes,
Jan Stevens and Jhon Quispe Vargas}

\maketitle

\vspace{-18pt}
\begin{abstract}
\noindent
We provide a new lower bound for the dimension of the moduli space of smooth pointed
curves with prescribed Weierstrass semigroup at the marked point, 
derived from the Deligne--Greuel formula and Pinkham's equivariant deformation theory.
Using Buchweitz's description of the first cohomology module of the cotangent complex for 
monomial curves, we show that our lower bound improves a recently one given by Pflueger. 
By allowing semigroups running over suitable families of symmetric semigroups of multiplicity six, 
we show that this new lower bound is attained,  and that the corresponding moduli spaces are non-empty
and of pure dimension.
\end{abstract}


\extrafootertext{{\em Keywords and Phrases:} Weierstrass Points, Moduli of Curves, Equivariant Deformation Theory.}

\extrafootertext{{\bf 2020 MSC:} 14H55, 14H10, 13D02, 14D15.}

\section{Introduction}
Given a smooth projective pointed curve $(\C,P)\in\mathcal{M}_{g,1}$ of genus $g>1$ defined over an algebraically
closed field $\k$, its associated Weiers\-trass semigroup $\N$ 
consists of the set of nonnegative integers $n$, called nongaps, such that there is a rational function on $\C$ whose
pole divisor is $nP$. Equivalently, $n$ is a nongap if and only if 
$\mathrm{H}^0(\C, \LO_{C}(n-1)P)\subsetneq\mathrm{H}^{0}(\C, \LO_{\C}(nP))$. The Riemann--Roch Theorem 
assures that the set of positive integers that are not in $\N$ (the set of gaps) has size exactly $g$.
In addition, we say that $P\in\C$ is a Weierstrass point if its associated Weierstrass semigroup is different from the 
ordinary one $\{0,g+1,g+2,\dots\}$.

For each numerical semigroup $\N$ of genus $g>1$, let $\M$ be 
the space parameterizing pointed smooth curves whose associa\-ted Weierstrass 
semigroup at the marked point is $\N$. It is well known that $\M$ can be empty, but if it is nonempty, 
since the $i$-th gap of a Weierstrass semigroup is 
an upper semicontinuous function, then  it is a locally  closed subspace of $\mathcal{M}_{g,1}$, so we get
get a stratification $\mathcal{M}_{g,1}=\sqcup_{\N} \M$, where $\N$ runs over all
the semigroups of genus $g$.

In this paper we focus on the problem of computing the dimension of $\M$. 
In characteristic zero there are two known  general bounds, namely the Rim--Vitulli upper
bound
\cite[\S 6]{RV77}, that dates back to the 70's, and more recently Pflueger's lower bound \cite{N16},
 \begin{equation}\label{lwb}
 \underbrace{3g-2-\ewt(\N)}_\text{Pflueger's bound}\leq \dim \M  \leq \underbrace{2g-2+\lambda(\N)}_\text{Rim--Vitulli bound},
 \end{equation} 
where $\ewt(\N)$ is the effective weight of $\N$ \cite[Def. 1.1]{N16} and $\lambda(\N)$ is the number of gaps $\l$ of $\N$ such that $\l+n\in\N$ 
whenever $n$ is a nongap. The methods used to get the above two bounds come from deformations of suitable curves:
while the Rim--Vitulli bound is derived from equivariant  deformations of monomial curves, Pflueger's bound
is derived by using deformations of stable curves and limit linear series. It is important to highlight that the two 
bounds in \eqref{lwb} coincide if and only if $\N$ is \textbf{negatively graded}, i.e. the first cohomology module $\mathrm{T}^{1}(\k[\N])$
of the cotangent complex associated to the semigroup algebra $\k[\N]$ is negatively graded, cf. \cite[Thm. 4.7]{RV77} and \cite[Prop. 2.11]{N16}.
 

The main result of this paper  (Theorem \ref{Stevens}) is that,   if
$\M$ is nonempty,
  \begin{equation}\label{comm}
 2g-2+\lambda(\N)-\dim\mathrm{T}^{1,+}(\k[\N])\leq\dim\mathrm{X},
 \end{equation}
for any irreducible component $\mathrm{X}$  of $\M$, 
where $\dim\mathrm{T}^{1,+}(\k[\N])$ 
stands for the dimension of the positive graded part of $\mathrm{T}^{1}(\k[\N])$. 
This bound is derived from
Pinkam's construction \cite{Pi74} of $\M$ and the Deligne--Greuel formula \cite{Del73,Gr82} for
smoothing components of curve singularities.
In addition,  using  Buchweitz' description \cite{Bu80} of $T^1$ for monomial curves, 
we show that the lower bound \eqref{comm} is not smaller than Pflueger's one (Proposition \ref{lbound}). 
 Hence, assuming that $\M$ is nonempty we can write
 \begin{equation}\label{ines}
3g-2-\ewt(\N)\leq 2g-2+\lambda(\N)-\dim\mathrm{T}^{1,+}(\k[\N])\leq \dim \M \leq 2g-2+\lambda(\N).
 \end{equation}
We show examples where our new lower bound does not attain the dimension of $\M$. 
Examples are also provided  where the first inequality in \eqref{ines} is strict.
In particular, for the semigroups  $\langle6,7,8\rangle$ and $\langle6,7,15\rangle$, Pflueger's bound does not provide the exact 
dimension of the $\M$, but our lower bound in Theorem \ref{Stevens} does.
These two particular
semigroups are symmetric, i.e. their largest gap are the biggest possible, namely $\l_g=2g-1$, suggesting that
symmetric semigroups can be an interesting class of semigroups to study the dimension of their associated moduli spaces $\M$.

In Section \ref{CScons} we briefly recall a  rather explicit construction of a compactification of $\M$ when $\N$ is
non-hyperelliptic and symmetric. This construction was given by  Stoehr \cite{St93}, then improved by Contiero--Stoehr \cite{c2013} and generalized
by Contiero--Fontes \cite{CF}.
All these constructions  can be viewed as a variant of Hauser's algorithm \cite{Hau83, Hau85} to compute versal deformation spaces, see also \cite{Stev13}.
Next, in Section \ref{Fam}, 
we use this construction to give explicit equations for  $\M$ 
when $\N$ runs over the following families of symmetric semigroups, 
\[
\langle6,3+6\tau,4+6\tau,7+6\tau,8+6\tau\rangle  \text{ and } \langle6,1+6\tau,2+6\tau,3+6\tau,14+6\tau\rangle
\]
and we show that  $\M$ is not empty in these cases by showing
that the associated affine monomial curves $\C_{\N}$ can be negatively smoothable, c.f. Theorems \ref{negsmoot1} and \ref{thm2}. Using a result of 
Contiero and Stoehr \cite{c2013} we find an upper bound for the dimension of $\M$,
which coincides with the lower bound, hence we are able to conclude that $\M$ is of pure dimension in these cases, Corollaries \ref{cor13} and \ref{boundfam2}.

\section{A new lower bound}\label{newbound}
The connection between the moduli space $\M$ and the negative part
of the deformation space of the monomial curve $\C_{\N}$ was first observed
by Pinkham \cite[Ch. 13]{Pi74}. Here $\N=\langle n_1,\dots,n_r\rangle$ 
is a numerical semigroup of 
genus $g> 1$ with semigroup ring   $\k[\N]:=\oplus_{n\in\N}\k\,t^{n}$
and  $\C_{\N}:=\Spec\k[\N]$ is its associated affine monomial curve. 
In the sequel we assume that $k$ is an algebraically closed field of
chacteristic zero.

Since $\C_{\N}$ has a unique singular point, there exists a versal deformation, 
c.f. \cite{Ar76}, say
%
\begin{equation}\label{diagrama}
\begin{matrix} 
\mathcal{X}_{t_0}\cong\C_{\N} & \longrightarrow & \mathcal{X} \\[3pt] 
\Big\downarrow & & \Big\downarrow\\[7pt]
\{t_0\}=\Spec\,\k & \longrightarrow & \mathcal{T}
\end{matrix}
\end{equation}
with $\mathcal{T}=\Spec A$, where $A$ is a local, complete noetherian $\k$-algebra.
Pinkham \cite{Pi74} showed that this deformation can be taken to be equivariant:
the  natural $\mathbb{G}_{m}$-action on $\C_{\N}$,
given by  $(\alpha, X_i)\mapsto \alpha^{n_i}X_i$,  
can be extended to the total  and
parameter spaces, $\mathcal{X}$ and  $\mathcal{T}$. This induces a natural
grading on the  tangent  space $T^1_{\C_{\N}}\cong T^{1}(\k[\N])$ to
$\mathcal{T}$. A deformation has negative weight $-e$ if it decreases the
weights of the equations of the curve and the corresponding deformation variable
has then (positive) weight $e$.
A numerical semigroup $\N$ is called \textit{negatively
graded} if  $T^{1}(\k[\N])$ has no positive graded part.

Let $\mathbf{I}$ be the ideal of $A$ generated by the images of the deformation variables  of negative weight, i.e. those elements
of $A$ that correspond to the elements in the positive graded part $T^{1,+}(\k[\N])$ of $T^{1}(\k[\N])$. 
Then $\mathcal{T}^-:=\Spec {A}/\mathbf{I}$  is the subspace of $\mathcal{T}$ in 
negative degrees and the restriction $\mathcal{X}^- \to \mathcal{T}^-$ is the 
versal deformation in negative degrees.
Both $\mathcal{X}^- $ and $ \mathcal{T}^-$ are defined by polynomials and we 
use the same symbols for the corresponding affine varieties.

The deformation $\mathcal{X}^- \to \mathcal{T}^-$ can be fiberwise compactified
to $\smash{\overline{\mathcal{X}}}^- \to \mathcal{T}^-$; each fibre is an integral curve
in a weighted projective space with one point $P$ at infinity and this is a point 
with semi-group $\N$.
All the fibres over a given
$\mathbb{G}_m$ orbit of $\mathcal{T}^-$ are isomorphic, and two fibres
are isomorphic if and only if they lie in the same orbit.
This is proved in \cite{Pi74} for smooth fibres and in general in the Appendix
of \cite{Lo84}.

Let $\C$ be a possible singular integral complete curve of arithmetic genus $g>1$ 
defined over 
$\k$. 
Given a smooth point $P$ of $\C$, let  $\N$ be the Weierstrass semigroup of $\C$ at $P$, that is the set 
of nonnegative integers  $n\in\N$ such that there is a rational function $x_n$ on $\C$ whose pole divisor $nP$. Consider the line  bundle $L=\LO(P)$ and
form the section ring $\mathcal{R} = \oplus_{i=0}^\infty H^0(\C,L^i)$.
This leads to an embedding of  $\C= \Proj \mathcal R$ in a weighted projective space, with coordinates $X_0,  \dots, X_r$ with $\deg X_0 = 1$.  The space $\Spec  \mathcal R$
is the corresponding quasi-cone in affine space. Setting $X_0 = 0$ defines the monomial 
curve $\C_\N$, all other fibres are isomorphic to $\C\setminus P$. In particular, if $\C$ is 
smooth, this construction defines a smoothing of $\C_\N$. 

\begin{thm}[\null{\cite[Thm. 13.9]{Pi74}}]\label{pinkhamthm}
Let $\mathcal{X}^- \to \mathcal{T}^-$ be the equivariant 
miniversal deformation in negative
degrees of the monomial curve $\C_\N$ for a given semigroup $\N$ and denote
by $(\mathcal{T}^-)_s$ the open subset of $\mathcal{T}^-$ given by the points
with smooth fibers. Then the moduli space $\M$ is isomorphic to the quotient
$\M=(\mathcal{T}^-)_s/\mathbb{G}_{m}$ of  $(\mathcal{T}^-)_s$ 
by the  $\mathbb{G}_{m}$-action.
\end{thm}

The assumption on the characteristic is essential for this result. For examples
in finite characteristic see \cite{Na16}.
The moduli space $\M$ is non-empty if and only if the monomial curve
$\C_\N$ can be smoothed negatively. 
Deligne \cite[Thm 2.27]{Del73} established a formula for the dimension 
of smoothing components of curve singularities in general (and in arbitrary characteristic). 
A \emph{smoothing component} 
is an irreducible component of the versal 
deformation space  whose fiber over its generic point is smooth. 
%
Deligne's formula simplifies for quasihomogeneous curves (in characteristic zero)  \cite{Gr82}.
\begin{thm}[Deligne--Greuel formula]\label{greuelformula}
For any smoothing component $\mathrm{E}$ of a quasihomogeneous curve $\Spec\LO$ 
\[
\dim\mathrm{E}=\mu+t-1\;.
\]
\end{thm} 
Here $\mu=2\delta-r+1$ is the Milnor number, where
$\delta:=\dim_{\k} \overline{\LO}/{\LO}$ the singularity degree of the curve at $P$,
$r$ is the number of branches and $t$ is the type
$\dim_k Ext^1_\LO(k,\LO)$. In particular, for Gorenstein curves
 $t=1$.
In the special case of monomial curves, where $\delta=g$ and $r=1$, 
the formula for the dimension of smoothing components already occurs in
\cite{Bu80}.
 In \cite[4.1.2]{Bu80}  a combinatorial  formula is stated for $t$, which shows that
 $t=\lambda(\N)$, the number of gaps $\l$ of $\N$ such that $\l+n\in\N$ whenever $n$ is a nongap. More formally,
let $\End(\N)=\{n\in \Nat\mid n+\N\setminus 0 \subset \N\}$. Then 
 $\lambda(\N)=\#(\End(\N) - \N) $. 
 
As each smoothing  component of $\mathcal{T}^-$ is contained in a smoothing 
component of the total versal deformation the Deligne--Greuel formula gives an upper bound
for the dimension of $\M=(\mathcal{T}^-)_s/\mathbb{G}_{m}$, first stated by
Rim and Vitulli \cite[\S 6]{RV77}.

\begin{thm}[Rim--Vitulli upper bound]\label{delupper} For any numerical semigroup $\N$
\[
\dim \M\leq 2g-2+\lambda(\N)\;.
\]
\end{thm}

This upper bound is attained: 
Rim and Vitulli showed, that if $\N$ is 
negatively graded, then $\C_\N$ is negatively smoothable 
\cite[Cor. 5.1]{RV77}, so $\dim \M=2g-2+\lambda(\N)$. 
A complete list of all negatively graded semigroups can also be found in \cite[Thm. 4.7]{RV77}.
In general the Rim--Vitulli bound is far from
being tight, for examples in low genus see Table \ref{tab1}. It also happens for
the families of semigroups in Section \ref{Fam}.

Let us assume  that $\C_{\N}$ can be smoothed negatively. 
Let $E^-$ be a smoothing component of $\mathcal{T}^-$, then there is smoothing
component $E$ of the versal deformation space of $\C_{\N}$ such that $E^-$ is
a component of $E\cap \mathcal{T}^-$. As $E\cap \mathcal{T}^-$ is obtained by
adding $\dim\mathrm{T}^{1,+}(\k[\N])$ linear equations to its defining equations, we have that the dimension of $Y$ is not smaller than
  $\dim E-\dim\mathrm{T}^{1,+}(\k[\N])$.  
With the Deligne--Greuel formula we obtain

\begin{thm}\label{Stevens}  Let $\N$ be a numerical semigroup $\N$ of genus bigger than $1$. 
If $\M$ is nonempty, then for any irreducible component $\mathrm{X}$ of $\M$  
\[
2g-2+\lambda(\N)-\dim\mathrm{T}^{1,+}\leq \dim\mathrm{X}\;.
\]
\end{thm}

To describe a dimension formula for the graded parts of $T^{1}(\k[\N])$ we start from a result due to Herzog that assures that the ideal of 
$C_{\N}:=\{(t^{a_1},\dots,t^{a_r})\,;\,t\in\k\}\subset\mathbb{A}^{r}$
can be generated by isobaric polynomials $F_i$ that are differences of two monomials, namely
$$F_i:=X_{1}^{\alpha_{i1}}\dots X_{r}^{\alpha_{ir}}-X_{1}^{\beta_{i1}}\dots X_{r}^{\beta_{ir}},$$
with $\alpha_i\cdot\beta_i=0$. As usual, the weight of $F_i$ is $d_i:=\sum_j n_j\alpha_{ij}=\sum_j n_j\beta_{ij}$.
For each $i$, let $v_i:=(\alpha_{i1}-\beta_{i1},\dots,\alpha_{ir}-\beta_{ir})$
be the vector in $\k^{r}$ induced by $F_i$.

\begin{thm}[cf. Thm. 2.2.1 of \cite{Bu80}]\label{t1busc}
Let $\k[\N]$ be the semigroup ring of a numerical semigroup $\N$.
For $\l\in\mathbb{Z}$ let $A_{\l}:=\{i\in\{1,\dots,r\}\mid n_i+\l\notin\N\}$
and let  $V_{\l}$ be the vector subspace of $\k^{r}$ generated by the vectors $v_i$
such that $d_i+\l\notin\N$. Then for $\l\notin\End(\N)$
\[
\dim T^{1}(\k[\N])_{\l}=\# A_\l-\dim V_{\l}-1\;
\]
while $\dim T^{1}(\k[\N])_s=0$ for  $s\in\End(\N)$ .
 \end{thm}

Pflueger produced  in \cite[Thm. 1.2]{N16} an upper bound for the codimension
of $\M$  as  locally closed subset of $\mathcal{M}_{g,1}$, improving  a
 bound  by Eisenbud and Harris in \cite{EH87}.
He introduced 
the \textit{effective weight} \cite[Def. 1.1]{N16} of a numerical semigroup $\N$ 
with a minimal system of generators,  $\N=\langle n_1,\dots,n_r\rangle$
\[
\ewt(\N):=\displaystyle\sum_{\mbox{gaps}\ \l_i}^{}(\#\ \mbox{generators}\ n_j<\l_i)\;.
\]
With this substitute for the classical weight $\mathrm{wt}(\N):=\sum\l_i-i$, which is equal to the sum over the gaps of $\#\ \mbox{nongaps}\ n<\l_i$,
appearing  in the Eisenbud--Harris bound Pflueger established the following result.

\begin{thm}[Pflueger's bound]\label{NPbound}
If the moduli space $\M$ is nonempty, and $X$ is any irreducible component of it, then
\[
\dim X\geq 3g-2-\ewt(\N)\;.
\]
\end{thm} 

We compare this bound with our bound from Theorem \ref{Stevens}.
Using the notation of  Theorem \ref{t1busc} we first give a different formula
for $\ewt(\N)$.
\begin{lem} Let $n_1<\cdots< n_r$ be a minimal system of generators for $\N$. We have
\[
\ewt(\N)=\sum_{\l \notin\End(\N)}\# A_{\l}.
\]
\end{lem}
\begin{proof}
By  definition  
$$\ewt(\N) = \sum_{\l\notin\N} \#\{n_i<\l \}=\# \{(n_i,\l_j)\mid n_i<\l\}.$$
On the other hand, $$\sum_{\l \notin \End(\N)} \# A_{\l} = \sum_{\l \notin \N} \# A_{\l},$$ 
where $\#A_{\l}$ is  $\# \{n_i+l \notin S\}$,
so this is the number of pairs $(n_i,\l)$ such that $n_i+\l$ is a gap.
Since the map $(n,\l)\mapsto(n,\l-n)$ is a bijection
from the first set of pairs to the second, their cardinality is the same.
\end{proof}   

 \begin{prop}\label{lbound}
 For any numerical semigroup $\N$ of genus $g\geq 1$ the bound of Theorem \ref{Stevens}
is not less then Pflueger's lower bound:
 $$3g-2-\ewt(\N)\leq 2g-2+\lambda(\N)-\dim T^{1,+}(\N)\;.$$
 \end{prop}
 \begin{proof}
Using Theorem \ref{t1busc} and the above lemma we obtain
\begin{equation*}
\dim T^{1,+}(k[\N])=
 \sum_{\l\notin \End(\N)} (\#A_\l-\dim V_\l-1) = 
\ewt(\N) - \sum_{\l\notin \End(\N)} \dim V_\l - \# (\Nat\setminus \End(\N))
\end{equation*} 
and $ \# (\Nat\setminus \End(\N))= \# (\Nat\setminus\N) - \# (\End(\N)\setminus\N)= 
g-\lambda(\N)$.
So
 $$3g-2-\ewt(\N) +  \sum_{\l\notin \End(\N)} \dim V_\l  =  2g-2+\lambda(\N)-\dim T^{1,+}(\N)\;.$$
\end{proof}

An example  where Pflueger's bound does not provide the exact dimension
of $\M$ is  given by Pflueger himself, cf. \cite[2F]{N16}. This example,
the symmetric semigroup $\N:=\langle6,7,8\rangle$ of genus $9$, fits in a more general context.
Let $\N$ be a symmetric semigroup generated by less than $5$ elements. Then 
the affine monomial curve $\Spec\ \k[\N]$ is a complete intersection, or if
$\N=\langle a_1,a_2,a_3,a_4\rangle$  a quasi-homogeneous version of 
Buchsbaum-Eisenbud's structure theorem
for Gorenstein ideals of codimension $3$ (see
\cite[p.\,466]{BE77}) applies. In both cases one can deduce that
$\Spec\ \k[\N]$
can be negatively smoothed without any obstructions (
\cite{W79} and \cite[Satz 7.1]{W80}), hence
\begin{equation}\label{truco}
\overline{\M} = \proj(\mathrm{T}^{1,-}(\k[\N]),
\end{equation} and therefore,
$\dim \M = \dim \proj(\mathrm{T}^{1,-}(\k[\N])$.
For $\N=\langle6,7,8\rangle$ one computes $\dim V_3 = 1$ and $\ewt(\N)=12$,
so $\dim\M=14$, Pflueger's bound gives $13$, while the Rim--Vitulli
bound provides $2g-1=17$.  In the same way, for
the symmetric semigroup $\langle6,7,15\rangle$ of genus $12$, Pflueger's bound gives
$17$, while $\dim V_2 = 1$ and $\dim\M=18$.

For each numerical semigroup of genus not bigger than $6$ the dimension of $\M$  
is equal to Pflueger's bound (see \cite[2C]{N16}), hence it is also equal to that given by Theorem \ref{Stevens}. In Table \ref{tab1} we collect for all 
all numerical semigroups of genus $g\leq 6$, which are not negatively graded,
the name of the semigroup in the list of Nakano \cite{Na08}, the gaps, 
the dimension of $\M$, which is also equal to both lower bounds, the
value of the Rim--Vitulli bound and $\dim T^{1,+}(\k[\N])$.

\begin{table}[htb]
\caption{non-negatively graded semigroups of genus $\leq 6$}
\label{tab1}
\begin{center}
\begin{tabular}{lcccc}
\cite{Na08}& gaps& $\dim\mathcal{M}_{g,1}^{\N}$& R--V & $\dim T^{1,+}$\\ \hline \hline
$N(5)_3 $ & 1, 2, 4, 5, 8& 9 & 10 & 1\\ \hline
$N(5)_5 $ & 1, 2, 3, 5, 7& 10 & 11 & 1\\ \hline
$N(5)_7 $ & 1, 2, 3, 6, 7& 9 &  10 & 1\\ \hline
$N(6)_3 $ & 1, 2, 4, 5, 7, 10&  11 & 12 & 1\\ \hline
$N(6)_4 $ & 1, 2, 4, 5, 8, 11&  10 & 11 & 1\\ \hline
$N(6)_6 $ & 1, 2, 3, 5, 6, 9&  12 & 13 & 1\\ \hline
$N(6)_7 $ & 1, 2, 3, 5, 6, 10& 11 & 12 & 1\\ \hline
$N(6)_8 $ & 1, 2, 3, 5, 7, 9& 11 & 13 & 2\\ \hline
$N(6)_9 $ & 1, 2, 3, 5, 7, 11& 10 & 11 & 1\\ \hline
$N(6)_{10} $ & 1, 2, 3, 6, 7, 11& 10 & 11 & 1\\ \hline
$N(6)_{12} $ & 1, 2, 3, 4, 6, 8& 13 & 14 & 1\\ \hline
$N(6)_{13} $ &  1, 2, 3, 4, 6, 9& 12 & 13 & 1\\ \hline
$N(6)_{15}  $ & 1, 2, 3, 4, 7, 8& 12 & 13 & 1\\ \hline
$N(6)_{16}  $ & 1, 2, 3, 4, 7, 9& 11 & 12 & 1\\ \hline
$N(6)_{17} $ &  1, 2, 3, 4, 8, 9& 10 & 12 & 2\\
\end{tabular}
\end{center}
\end{table}

\begin{rmk}
Nakano \cite{Na08} computed $\M$ using Pinkham's theorem  by
determining the base space of the miniversal deformation in negative degrees
of the monomial curve $\C_\N$. In all cases he succeeded except one
$\M$  has the structure of a projective quasi-cone over $\proj^1\times\proj^3$.
In case $N(6)_8 $ with semigroup $\langle 4,6,11,13\rangle$ the base space
of the total versal deformation is, after a coordinate transformation,  given by
\[
\operatorname{Rk}
\begin{pmatrix}
a_9 & a_{11} & a_{16} & a_{18}\\
a_{-1}&a_1 &a_6 & a_8
\end{pmatrix}
\leq1\;,
\]
where the entries are deformation variables indexed by their weight. The remaining
variables have weights $14,12,10,8,6,4,4,2$ and $-3$.
So this base space is irreducible, but the intersection with $\mathcal{T}^-$ 
is given by $a_{-1}=0$ and consists two components, one smoothing component,
and the other not: the general fiber over the component 
$a_1=a_6= a_8=0$ is a curve with a double point.
\end{rmk}

\begin{rmk}
There are semigroups where our new lower bound given by
Theorem \ref{Stevens} is not attained. For example, if $\M$ is not of pure dimension,
then the dimension of the biggest component does not attain our new bound.
The first examples are given  by Pflueger \cite[Thm 1.1]{NP2}
in his  detailed study of the moduli variety 
$\M$ when $\N$ is a \textit{Castelnuovo semigroup},  a semigroup generated
by consecutive suitable positive integers: $S_{r,d}:=\langle d-r+1,\dots,d\rangle$,with $d\geq2r-1$. Castelnuovo semigroups  also provide many examples where 
Pflueger's bound (Theorem \ref{NPbound}) are not attained, see \cite[Prop 6.3]{NP2}.

In addition, there are symmetric semigroups where the lower bound
 $2g-2+\lambda(\N)-\dim T^{1,+}(\k[\N])=2g-1-\dim\mathrm{T}^{1,+}$ is negative. 
For example,  for the symmetric semigroups $\langle29,30,\dots,42,57\rangle$, 
$\langle31,32,\dots,45,61\rangle$ and $\langle33,34,\dots,48,64\rangle$, 
of genus $43$, $46$ and $49$,   the number 
$2g-1-\dim\mathrm{T}^{1,+}$ is
 $-6$, $-14$ and $-23$, respectively. 
\end{rmk}

\section{Weierstrass points on Gorenstein curves}\label{CScons}

Let $\C$ be a possible singular integral complete Gorenstein curve of arithmetic genus $g>1$ defined
over an algebraically closed field $\k$. Given a smooth point $P$ of $\C$, let 
$\N$ be the Weierstrass semigroup of $\C$ at $P$, that is the set of nonnegative integers 
$n\in\N$ such that there is a rational function $x_n$ on $\C$ whose pole divisor $nP$. Let us assume 
that the semigroup $\N$ is symmetric, i.e.  $n\in\N$ if and only if $\l_g-n\notin\N$,
where  $\l_g$ is the last gap. Equivalently, $\l_g$ is the largest possible, $\l_g=2g-1$.
A basis for the vector space $H^{0}(\C,\mathcal{O}_{\C}((2g-2)P))$ is
$\{x_{n_0}, x_{n_1},\ldots, x_{n_{g-1}}\}$, 
and thus $\mathcal{O}_{\C}((2g-2)P)\cong\omega$, where $\omega$ is 
the dualizing sheaf of $\C$. By assuming that $\C$ is nonhyperelliptic, 
the canonical morphism
\[
(x_{n_0}: x_{n_1}:\ldots: x_{n_{g-1}}):\C\hookrightarrow\mathbb{P}^{g-1}
\] 
is an embedding.
Thus $\C$ becomes a curve of genus $g$ of degree $2g-2$ in $\mathbb{P}^{g-1}$ and the 
integers $\l_i-1$ are the contact orders of the curve with the 
hyperplanes at $P=(0:\ldots:0:1)$. 
Conversely, any nonhyperelliptic symmetric
semigroup $\N$ can be realized as the Weierstrass semigroup of the canonical Gorenstein
 \textit{monomial curve}
\[
\C_{\N}:=\{(s^{n_0}t^{\l_g-1}: s^{n_1}t^{\l_{g-1}-1}:\ldots:s^{n_{g-2}}t^{\l_2-1}: s^{n_{g-1}}t^{\l_1-1})\,\vert\,(s:t)\in\mathbb{P}^1\}\subset \mathbb{P}^{g-1}\,
\]
at its unique point $P=(0:\dots:0:1)$ at the infinity. 


Now we recall the construction of a compactification of $\M$, when $\N$ is symmetric, that was first introduced by Stoehr \cite{St93}, 
then improved by Contiero--Stoehr \cite{c2013} and generalized
by Contiero-Fontes \cite{CF}.
Let us start with a pointed canonical Gorenstein curve $(\C,P)$ whose Weierstrass semigroup $\N$ at $P$ is symmetric.
We know from  \cite[Theorem 1.3]{O91} that each nongap $s\leq4g-4$ can be written as a 
sum of two others nongaps, namely
\[
s=a_s+b_s,\quad a_s\leq b_s\leq2g-2.
\]
By choosing $a_s$ the smallest possible, the $3g-3$ rational functions $x_{a_s}x_{b_s}$ form a  $P$-hermitian 
basis of the global sections $H^{0}(\C,\mathcal{O}_{\C}(2(2g-2)P))$ of the 
bicanonical divisor. 
The  homomorphism
\[ \k[X_{n_0},\ldots,X_{n_{g-1}}]_2\longrightarrow H^{0}(\C,\mathcal{O}_{\C}(2(2g-2)P))
\]
induced by the substitutions $X_{n_{i}}\longmapsto x_{n_i}$ is surjective
and the kernel is the vector space of quadratic 
forms in  the ideal of $\C\subset\proj^{g-1}$.

%
Now, given a nongap $s\leq4g-4$, let us consider all the partitions of $s$ as sum of 
two nongaps not greater than $2g-2$,
\[
s=a_{si}+b_{si},\text{ with }a_{si}\leq b_{si}\ (i=0,\ldots, \nu_s), \text{ where } a_{s0}:=a_s.
\]
Hence, given a nongap $s\leq4g-4$ and $i=1,\ldots,\nu_s$  we can write
\[
x_{a_{si}}x_{b_{si}}=\displaystyle\sum_{n=0}^{s}c_{sin}x_{a_{n}}x_{b_{n}},
\]
where $a_n$ and $b_n$ are nongaps of $\N$ whose sum is equal to $n$, and $c_{sin}$ are suitable constants in $\k$.
By normalizing 
the coefficients $c_{sis}=1$, it follows that the 
${g+1\choose 2}-(3g-3)=(g-2)(g-3)/2$ quadratic forms 
\[
F_{si}=X_{a_{si}}X_{b_{si}}-X_{a_s}X_{b_s}-\sum_{n=0}^{s-1}c_{sin}X_{a_{n}}X_{b_{n}}
\] 
vanish identically on the canonical curve $\C$, where the coefficients 
$c_{sin}$ are uniquely determined constants. They are linearly independent, 
hence they form a basis for the space of quadratic relations in $I(\C)$.
 
We need to make some assumptions on the symmetric semigroup $\N$ to assure 
that the ideal $I(\C)$ is generated by its quadratic relations $F_{si}$.
Precisely, we have to assume that $\N$ satisfies $3<n_1<g$ and $\N\neq\langle4, 5\rangle$. 
According to \cite[Lemma 3.1]{CF}, both the conditions $n_1\neq3$ and $n_1\neq g$ 
avoid possible trigonal Gorenstein curves whose Weierstrass semigroup at $P$ 
equal to $\N=\langle3, g+1\rangle$ and 
$\N=\langle g, g+1,\ldots, 2g-2\rangle$, respectively. This two avoided cases are also treated
by  similar techniques in \cite{CF}, but suitable cubic forms are  required to compute the
ideal of the canonical Gorenstein curve $\C$. 
So, making the above assumptions on the semigroup $\N$, it follows by the Enriques--Babbage 
Theorem  that  $\C$ is nontrigonal and  not isomorphic to a plane quintic. 
Hence the ideal of $\C$ is generated by the $(g-2)(g-3)/2$ quadratic
forms $F_{si}$, c.f. \cite[Theorem 2.5]{c2013}.

On the other hand, starting with a symmetric semigroup $\N$ with $3<n_1<g$ and $\N\neq\langle4, 5\rangle$,
let us introduce the following $(g-2)(g-3)/2$ quadratic forms 
\begin{equation}\label{unfold}
F_{si}=X_{a_{si}}X_{b_{si}}-X_{a_s}X_{b_s}-\displaystyle\sum_{n=0}^{s-1}c_{sin}X_{a_{n}}X_{b_{n}},
\end{equation} where $c_{sin}$ are constants to be determined in order that the intersection of the 
$V(F_{si})$ in $\proj^{g-1}$ is a canonical Gorenstein curve of genus $g$ whose Weierstrass semigroup at $P$ is 
$\N$. We attach to the variable $X_n$ the weight $n$ and to 
each coefficient $c_{sin}$ the weight $s-n$. If we consider $F_{si}$ as a polynomial
expression, not only in the variables $X_n$, but also in the coefficients $c_{sin}$, then it becomes quasi-homogeneous of weight $s$.

Since the coordinates functions $x_n$, $n\in\N$ 
and $n\leq2g-2$, are not uniquely determined by their pole divisor $nP$,
we may transform
\[
X_{n_i}\longmapsto X_{n_i}+\sum_{j=0}^{i-1}c_{n_in_{i-j}}X_{n_{i-j}},
\]
for each $i=1,\ldots, g-1$, and so we can normalize 
$\frac{1}{2}g(g-1)$ of the coefficients $c_{sin}$ to be zero, see 
\cite[Proposition 3.1]{St93}. Due to these normalizations and the 
normalizations of the coefficients $c_{sin}=1$ with $n=s$, the only freedom left to us 
is to transform $x_{n_i}\mapsto c^{n_i}x_{n_i}$ for $i=1,\ldots, g-1$.

Let 
\[
F_{si}^{(0)}:=X_{a_{si}}X_{b_{si}}-X_{a_s}X_{b_s}
\] 
be the quadratic forms that generate the ideal of the canonical monomial curve
$\C_{\N}$, cf. \cite[Lemma 2.2]{c2013}. 
One of the keys to construct a compactification of $\M$ is the following lemma.

\begin{syzlem}[cf. \cite{c2013}]\label{syzlem}
For each one of the $\frac{1}{2}(g-2)(g-5)$ quadratic binomials $\F{s'i'}^{(0)}$
different from 
$\F{n_i+2g-2,1}^{(0)}$ $(i=0,\dots,g-3)$, there is
a linear syzygy of the form
\begin{equation}\label{linsyzy}
X_{2g-2}\F{s'i'}^{(0)}+\sum_{nsi}\varepsilon_{nsi}^{(s'i')}X_{n}\F{si}^{(0)}=0
\end{equation}
where the coefficients $\varepsilon_{nsi}^{(s'i')}$ are integers equal to $1$,
$-1$ or $0$, and the sum is taken over the
nongaps $n<2g-2$ and the double indexes $si$ such that $n+s=2g-2+s'$.
\end{syzlem}

The explicit construction of a compactification of $\M$ starts by replacing 
the initials binomials $\F{s'i'}^{(0)}$ and $\F{si}^{(0)}$ in equation \eqref{linsyzy} 
by the corresponding 
forms $\F{s'i'}$  and
$\F{si}$ displayed in equation \eqref{unfold}, obtaining a 
linear combination of cubic monomials of weight $<s'+2g-2$. By virtue of \cite[Lemma 2.4]{c2013}  and its proof
this linear combination of cubic monomials admits the following decomposition.
\begin{equation*}
 X_{2g-2}\F{s'i'}+\sum_{nsi}\varepsilon_{nsi}^{(s'i')}X_{n}\F{si}=
 \sum_{nsi}\eta_{nsi}^{(s'i')}X_{n}\F{si}+R_{s'i'}
\end{equation*}
where the sum on the right hand side is taken over the nongaps $n\leq 2g-2$ and
the double indexes $si$ with $n+s<s'+2g-2$, the coefficients
$\eta_{nsi}^{(s'i')}$ are constants, and where $R_{s'i'}$ is a linear
combination of cubic monomials of pairwise different weights $<s'+2g-2$.

For each nongap $m<s'+2g-2$, let $\varrho_{s'i'm}$ be the unique coefficient
of $R_{s'i'}$ of weight $m$. It is a quasi-homogeneous polynomial expression
of weight $s'+ 2g - 2 - m$ in the coefficients $c_{sin}$.

\begin{thm}\cite[Theorem 2.6]{c2013}\label{teo3}
Let $\N$ be a symmetric semigroup of genus $g$ sa\-tis\-fying $3<n_1<g$ and 
$\N\neq\langle4, 5\rangle$.
The isomorphism classes of the pointed complete integral Gorenstein curves with Weierstrass 
semigroup $\N$ correspond bijectively to the orbits of the $\mathbb{G}_m(\k)$-action
\begin{equation*}
(c;\ldots,c_{sin},\ldots)\longmapsto(\ldots,c^{s-n}c_{sin},\ldots)
\end{equation*}
on the  affine quasi-cone of the vectors whose coordinates are the 
coefficients $c_{sin}$ of the normalized quadratic $F_{si}$ satisfying 
the quasi-homogeneous equations $\varrho_{s'i'm}=0$. 
\end{thm}

\begin{rmk}
This construction  can be viewed as a variant of Hauser's algorithm 
to compute versal deformation spaces \cite{Hau83, Hau85} , see also \cite{Stev13}.
The standard approach in deformation theory is to successively lift infinitesimal deformations
to higher order, collecting the obstructions at each stage. In Hauser's algorithm 
the defining equations of a singularity are first unfolded,  an unobstructed problem,
and only  then flatness is imposed by lifting relations using a division procedure.
The unique remainder leads to equations on the unfolding parameters.
In general there are infinitely many such parameters, but if the singularity is
quasi-homogeneous, we can restrict to the non-positive part of the unfolding
and we obtain equations in finitely many variables for the versal deformation
in non-positive weight.

By  setting $X_0=1$ we see that the equations $\varrho_{s'i'm}=0$ above
give the miniversal deformation in negative weight of the affine monomial
curve $\C_\N$ in a non-minimal embedding in $\AF^{g-1}$.
The compactification of the moduli space $\M$ in Theorem \ref{teo3}
corresponds to Pinkham's theorem \ref{pinkhamthm}. 

The advantage of working with a non-minimal embedding is that all equations
become quadratic, but at the price of an increase in the number of variables. 
\end{rmk}

\section{Families of symmetric $6$-semigroups}\label{Fam}

In this section we apply the techniques briefly described in the above section to deal
with families of symmetric semigroups. If the symmetric semigroup is generated by less
than five elements, the dimension of the moduli variety $\M$ is well known, as  noted in
Section \ref{newbound}. So, we must
consider symmetric semigroups of multiplicity greater than $5$, just because a symmetric
semigroup of multiplicity $m$ can be generated by $m-1$ elements. 

\subsection{A first family}
 For each positive integer $\tau$ consider the semigroup
\begin{align}\notag
    \N&=\langle 6, 3+6\tau, 4+6\tau, 7+6\tau, 8+6\tau\rangle\\
    &=6\mathbb{N}\sqcup\bigsqcup_{j\in\{3, 4, 7, 8\}}^{}(j+6\tau+6\mathbb{N})\sqcup(11+12\tau+6\mathbb{N}).\label{semigp}
\end{align}
of multiplicity $6$ minimally generated by five elements. Counting the number of gaps of 
$\N$ and picking up the largest one, we find
\[
g=3+6\tau \quad \text{and}\quad \l_g=12\tau+5=2g-1,
\]
showing that $\N$ is a symmetric semigroup. 

Let $\mathcal{C}$ be a complete integral Gorenstein 
curve and $P$ be a smooth point of $\C$ whose Weierstrass semigroup at $P$ is $\N$. 
For each $n\in\N$, let $x_n$ be  a rational function on $\C$ with pole divisor $nP$. 
We abbreviate
\[
x:=x_6\quad \text{and}\quad y_j:=x_{j+6\tau}\ (j=3, 4, 7, 8)
\]
and normalize 
\[
x_{6i}=x^i \quad \text{and} \quad x_{j+6\tau+6i}=x^iy_j, \  \forall\, i\geq1.
\]

By considering the above the normalizations we see that the functions
$1,x, y_3, y_4, y_7, y_8$ generate the ring $\bigoplus_k H^{0}(\C, kP)$
and therefore induce an embedding 
\[
  (1:x: y_3: y_4: y_7: y_8):\C\hookrightarrow\mathbb{P}:=\mathbb{P}(1,6,3+6\tau,4+6\tau,
  7+6\tau,8+6\tau)
\]
into a weighted projective space $\proj$, whose image we call $\mathcal{D}$.
Instead of studying the ideal of the canonical curve $\C\subset \proj^{g-1}$, 
which has $(g-2)(g-3)/2$ quadratic generators,
we study the ideal (and the relations between its generators) of the curve 
$\mathcal{D}\subset\proj$. The advantage is that the number of generators of the ideal of $\mathcal{D}$  does not depends on the genus $g$, see Lemma \ref{lem11} below.

We work in the affine chart on  $\proj$ obtained by setting the first
coordinate equal to $1$.
Let $X, Y_3, Y_4, Y_7, Y_8$ be indeterminates whose weights are
 $6, 3+6\tau, 4+6\tau, 7+6\tau, 8+6\tau$, respectively. For each $n\in\N$, 
 we introduce a monomial $Z_n$ of weight $n$ as follows
 \[
 Z_{6i}=X^i,\ Z_{j+6\tau+6i}=Y_jX^i\text{ and }Z_{11+12\tau+6i}=Y_3Y_8X^i.
 \]
By writing the nine products $y_iy_j, (i,j)\neq(3,8)$ as linear combination of the basis 
 elements, we obtain nine polynomials in the indeterminates $X, Y_3, Y_4, Y_7, Y_8$ that vanish 
 identically on the affine  curve $\mathcal{D}\cap\mathbb{A}^5$, say 
  \begin{equation}\label{fam1pol}
  F_i=F_i^{(0)}+\sum_{j=0}^{12\tau+i}f_{ij}Z_{12\tau+i-j}\qquad (i=6, 7, 8,10,11, 12, 14, 15,16),
   \end{equation}
 where \begin{equation}\label{initial1}
\begin{array}{lll}
   F_6^{(0)}= Y_3^2-X^{2\tau+1},&  F_7^{(0)}= Y_3Y_4-X^{\tau}Y_7,&  F_8^{(0)}= Y_4^2-X^{\tau}Y_8,\\
    F_{10}^{(0)}= Y_3Y_7-X^{\tau+1}Y_4,& F_{11}^{(0)}=Y_4Y_7-Y_3Y_8,& F_{12}^{(0)}= Y_4Y_8-X^{2\tau+2},\\
     F_{14}^{(0)}= Y_7^2-X^{\tau+1}Y_8,& F_{15}^{(0)}=Y_7Y_8-X^{\tau+2}Y_3,& F_{16}^{(0)}=Y_8^2-X^{\tau+2}Y_4,
\end{array}
\end{equation}
and the index $j$ only varies through integers with $12\tau+i-j\in\N$. The 
proof of the next lemma is very similar to \cite[Lemma 4.1]{c2013}.  
    
  \begin{lem}\label{lem11} The ideal of the affine 
  curve $\mathcal{D}\cap\mathbb{A}^5$ is equal to the ideal $\mathcal{I}$ generated by the 
  forms $F_i \ (i=6, 7, 8,10,11, 12, 14, 15, 16)$. In particular the ideal of the affine monomial curve 
  \[
\C_{\N}=\{(t^6, t^{3+6\tau}, t^{4+6\tau}, t^{7+6\tau}, t^{8+6\tau}) \mid t\in\k\}
\]
is generated by the initial forms $F_i^{(0)}\ (i=6, 7, 8, 10, 11, 12, 14, 15, 16)$.
   \end{lem}
\begin{proof}
It is clear that $\mathcal{I}\subseteq I(\mathcal{D}\cap\mathbb{A}^5)$. Let $f$ be a polynomial 
in the variables $X, Y_3, Y_4, Y_7, Y_8$. By applying induction 
on the degree of $f$ in the indeterminate $Y_3, Y_4, Y_7, Y_8$ we note that, 
module the ideal generated by the 
nine forms $F_i$, the monomials of this polynomial $f$ are not divisible by the 
nine products $Y_iY_j$, $(i,j)\neq(3,8)$, hence the class of $f$ is a sum $\sum c_nZ_n$ 
of monomials $Z_n$ of pairwise different weights with $n\in\N$ and $c_n\in\k$. Thus the 
polynomial $f$ belongs to $I(\mathcal{D}\cap\mathbb{A}^5)$
if and only if the linear combination $\sum c_nZ_n$ vanishes identically on the curve 
$\mathcal{D}\cap\mathbb{A}^5$ and by taking the corresponding linear combination 
$\sum c_nx_n$ of rational functions on $\k(\C)$ we have $c_n=0$ for 
each $n\in\N$, hence $f$ belongs to $\mathcal{I}$.     
\end{proof}
   
Now let us invert the above situation. Given the above fixed symmetric semigroup $\N$,  we introduce  nine isobaric polynomials like in \eqref{fam1pol}:
     \begin{equation}\label{fam1pol2}
  F_i=F_i^{(0)}+\sum_{j=0}^{12\tau+i}f_{ij}Z_{12\tau+i-j}\qquad (i=6, 7, 8,10,11, 12, 14, 15,16)
   \end{equation}
in  the polynomial ring $\k[X,Y_3,Y_4,Y_7,Y_8]$, where the coefficients $f_{ij}$ have  weight $j$. 
We want relations on the 
$f_{ij} $ 
in order that they give rise to a Gorenstein curve in $\proj$ 
whose Weierstrass semigroup is $\N$ at 
the marked point $P=(0:0:0:0:0:1)$. Equivalently, we want equations on the
$f_{ij} $ such that the polynomials \eqref{fam1pol2} define the miniversal
deformation of the affine monomial curve $\C_\N$.

To  express these conditions in a concise manner we first note that the coordinate ring
of the monomial curve $\C_{\N}$ is a free $k[X]$-module
generated by $1$, $Y_3$, $Y_4$, $Y_7$, $Y_8$ and $Y_3Y_8$; this corresponds to the decomposition of the semigroup in \eqref{semigp}. We write the polynomials
\eqref{fam1pol2} as polynomials in the $Y_i$ with coefficients in $k[X]$:

\begin{equation}
\label{fam1pol3}
\begin{aligned}
F_{6} &= \Y(3)^2-X^{2\tau+1}+\f(6)(2)\Y(4)+\f(6)(3)\Y(3)+\f(6)(4)\Y(8)+\f(6)(5)\Y(7)+\f(6)(6)\\
F_{7} &= \Y(3)\Y(4)-X^{\tau}\Y(7)+\f(7)(1)+\f(7)(3)\Y(4)+\f(7)(4)\Y(3)+\f(7)(5)\Y(8)+\f(7)(6)\Y(7)\\
F_{8} &= \Y(4)^2-X^{\tau}\Y(8)+\f(8)(1)\Y(7)+\f(8)(2)+\f(8)(4)\Y(4)+\f(8)(5)\Y(3)+\f(8)(6)\Y(8)\\
F_{10} &= \Y(3)\Y(7)-X^{\tau+1}\Y(4)+\f(10)(1)\Y(3)+\f(10)(2)\Y(8)+\f(10)(3)\Y(7)+\f(10)(4)+\f(10)(6)\Y(4)\\
F_{11} &= \Y(4)\Y(7)-\Y(3)\Y(8)+\f(11)(1)\Y(4)+\f(11)(2)\Y(3)+\f(11)(3)\Y(8)+\f(11)(4)\Y(7)+\f(11)(5)\\
F_{12} &= \Y(4)\Y(8)-X^{2\tau+2}+\f(12)(1)\Y(3)\Y(8)+\f(12)(2)\Y(4)+\f(12)(3)\Y(3)+\f(12)(4)\Y(8)+\f(12)(5)\Y(7)+\f(12)(6)\\
F_{14} &= \Y(7)^2-X^{\tau+1}\Y(8)+\f(14)(1)\Y(7)+\f(14)(2)+\f(14)(3)\Y(3)\Y(8)+\f(14)(4)\Y(4)+\f(14)(5)\Y(3)+\f(14)(6)\Y(8)\\
F_{15} &= \Y(7)\Y(8)-X^{\tau+2}\Y(3)+\f(15)(1)\Y(8)+\f(15)(2)\Y(7)+\f(15)(3)+\f(15)(4)\Y(3)\Y(8)+\f(15)(5)\Y(4)+\f(15)(6)\Y(3)\\
F_{16} &= \Y(8)^2-X^{\tau+2}\Y(4)+\f(16)(1)\Y(3)+\f(16)(2)\Y(8)+\f(16)(3)\Y(7)+\f(16)(4)+\f(16)(5)\Y(3)\Y(8)+\f(16)(6)\Y(4)
\end{aligned}
\end{equation}
Here  $f_i^{(j)}=\sum_{k=0}^{\rho}f_{i,j+6k}X^{\rho-k}$ with $\rho$ determined
by the condition that the polynomials are isobaric. If $i-j=6\varepsilon$, then
$\rho=2\tau+\varepsilon$; if $i-j=6\varepsilon+1$ or $i-j=6\varepsilon+2$ 
then $\rho=\tau-1+\varepsilon$; if $i-j=6\varepsilon+3$ or $i-j=6\varepsilon+4$ 
then $\rho=\tau+\varepsilon$ and finally $\rho=0$ for $i-j=11$.

Some of the coefficients can be made to vanish by homogeneous
coordinate transformations of the form
 \begin{align*}
 X&\mapsto X+c_6\\
  Y_3&\mapsto Y_3+\textstyle\sum_{i=0}^{\tau}c_{3+6\tau}X^{\tau-i}\\
Y_4&\mapsto Y_4+c_1Y_3+\textstyle\sum_{i=0}^{\tau}c_{4+6\tau}X^{\tau-i}\\
 Y_7&\mapsto Y_7+c_3Y_4+c_4Y_3+\textstyle\sum_{i=0}^{\tau+1}c_{1+6\tau}X^{\tau+1-i}\\
 Y_8&\mapsto Y_8+c_1'Y_7+c_4'Y_4+c_5Y_3+\textstyle\sum_{i=0}^{\tau+1}c_{2+6\tau}X^{\tau+1-i},\\
 \end{align*}
where $c_i$ and  $c_i'$ are constants of weight $i$. We normalize
\begin{equation}\label{normfam1}
\begin{array}{r@{}l@{}l@{}l}
f_{7}^{(3)}&{}=f_{11}^{(4)}&{}=f_{10}^{(1)}&{}=f_{11}^{(2)}=0,\\
f_{12}^{(1)}&{}=f_{14}^{(3)}&{}=f_{15}^{(4)} &{}=f_{16}^{(5)} = 0.
\end{array}
\end{equation}
There are still three normalizations left, which can be use to make the first coefficient in an $f_i^{(j)}$ to zero: we can take $f_{8,1}=f_{12,4}=f_{8,6}=0$.
Then  the only freedom left is given by the 
 $\mathbb{G}_m(\k)$-action.

 The Syzygy Lemma applied to the nine initial isobaric forms in \eqref{initial1},
 give rises to only eight syzygies of the affine monomial curve $\C_\N$, namely
 \begin{equation}\label{syz}
 \begin{array}{l}
 Y_4F_{6}^{(0)}-Y_3F_7^{(0)}-X^{\tau}F_{10}^{(0)}=0 \\
 Y_4F_{7}^{(0)}-Y_3F_{8}^{(0)}+X^{\tau}F_{10}^{(0)}=0\\
 Y_4F_{10}^{(0)}-Y_7F_7^{(0)}+X^{\tau+1}F_{8}^{(0)}-X^{\tau}F_{14}^{(0)}=0\\
 Y_4F_{11}^{(0)}-Y_7F_8^{(0)}+Y_8F_7^{(0)}=0\\
 Y_4F_{12}^{(0)}-Y_{8}F_8^{(0)}-X^{\tau}F_{16}^{(0)}=0\\
 Y_4F_{14}^{(0)}-Y_{8}F_{10}^{(0)}-Y_{7}F_{11}^{(0)}=0\\
 Y_4F_{15}^{(0)}-Y_7F_{12}^{(0)}+X^{\tau+2}F_{7}^{(0)}=0\\
 Y_4F_{16}^{(0)}-Y_8F_{12}^{(0)}+X^{\tau+2}F_{8}^{(0)}=0.\\
 \end{array}
 \end{equation}
 The total number of syzygies is 16. The other syzygies can easily be found by lifting the
syzygies of the zero-dimensional ring obtained by setting $X=0$.
 
Now we replace in the syzygies the initial isobaric forms $F_{i}^{(0)}$  by the
associated unfolded $F_{i}$  in \eqref{fam1pol3}, with the normalizations
\eqref{normfam1}. 
Next we apply the division 
 algorithm to  quadratic monomials in the $Y_i$. 
For the first syzygy we get:
\begin{multline}
\label{syz1}
\quad Y_4F_{6}-Y_3F_7-X^{\tau}F_{10}
-\f(10)(2)F_8+\f(11)(3)F_7-\f(6)(4)F_{12}+\f(12)(4)F_6
-\f(7)(5)F_{11}+\f(8)(6)F_{10}= {}\\
R_5 Y_3Y_8 + R_2 Y_8 + R_3Y_7+ R_6 Y_4 + R_1 Y_3 +R_4
 \quad
\end{multline}
where
\begin{equation}
\label{Reqns}
\begin{aligned}
R_5&=\f(6)(5)-\f(7)(5)\\
R_2&=X^\tau \f(6)(2)-X^\tau \f(10)(2)+\f(6)(4)\f(7)(4)-\f(6)(3)\f(7)(5)-\f(6)(2)\f(8)(6)+\f(7)(6)\f(10)(2)-\f(6)(5)\f(11)(3)-\f(6)(4)\f(12)(4)\\
R_3&=X^\tau \f(6)(3)-X^\tau \f(10)(3)+\f(6)(5)\f(7)(4)-\f(6)(3)\f(7)(6)-\f(6)(2)\f(8)(1)+\f(7)(6)\f(10)(3)-\f(6)(4)\f(12)(5)\\
R_6&=\f(6)(6)-X^{\tau+1} \f(7)(6)-X^\tau \f(10)(6)+\f(6)(2)\f(7)(4)-\f(6)(2)\f(8)(4)+\f(7)(6)\f(10)(6)-\f(6)(5)\f(11)(1)-\f(6)(4)\f(12)(2)\\
R_1&=-\f(7)(1)-\f(6)(2)\f(8)(5)-\f(6)(4)\f(12)(3)\\
R_4&=X^{2\tau+2}\f(6)(4)-X^{2\tau+1}\f(7)(4)-X^\tau \f(10)(4)
 -\f(6)(3)\f(7)(1)+\f(6)(6)\f(7)(4)-\f(6)(2)\f(8)(2)+\f(7)(6)\f(10)(4)\\ &
 \qquad\qquad -\f(6)(5)\f(11)(5)-\f(6)(4)\f(12)(6)
\end{aligned}
\end{equation}
The condition that the syzygy between the $F_{i}^{(0)}$  lifts to a syzygy between
the $F_{i}$ is that the coefficients $R_i$ vanish. This in turn leads to
quasi-homogeneous equations between the coefficients $f_{ij}$. We observe that the
vanishing of
the coefficient $R_5$ of $Y_3Y_8$ gives the linear equation
$\f(6)(5)=\f(7)(5)$. The same happens for the other 15 syzygies.
We get the 16 linear equations
\begin{align*}
\f(15)(1)&=0,  & \f(11)(1)&=\f(14)(1), \\
\f(12)(2)&=\f(15)(2)=\f(16)(2),  &\f(6)(2)&=\f(10)(2), \\
\f(10)(3)&=0,  &\f(6)(3)&=-\f(11)(3),  &\f(16)(3)&=\f(12)(3), \\
\f(7)(4)&=\f(8)(4)=\f(12)(4),\\
\f(15)(5)&=\f(14)(5),  &\f(6)(5)&=\f(7)(5)& \f(8)(5)&=\f(12)(5), \\
\f(7)(6)&=\f(8)(6)& \f(10)(6)&=\f(14)(6)&\f(16)(6)&=\f(15)(6)
\end{align*}
We use these equations to reduce the number of variables. Together with 
the normalizations \eqref{normfam1} our polynomials \eqref{fam1pol3}
reduce then to
\begin{equation}
\label{fam1pol4}
\begin{aligned}
F_{6} &= \Y(3)^2-X^{2\tau+1}+\f(10)(2)\Y(4)-\f(11)(3)\Y(3)+\f(6)(4)\Y(8)+\f(7)(5)\Y(7)+\f(6)(6)\\
F_{7} &= \Y(3)\Y(4)-X^{\tau}\Y(7)+\f(7)(1)+\f(12)(4)\Y(3)+\f(7)(5)\Y(8)+\f(8)(6)\Y(7)\\
F_{8} &= \Y(4)^2-X^{\tau}\Y(8)+\f(8)(1)\Y(7)+\f(8)(2)+\f(12)(4)\Y(4)+\f(12)(5)\Y(3)+\f(8)(6)\Y(8)\\
F_{10} &= \Y(3)\Y(7)-X^{\tau+1}\Y(4)+\f(10)(2)\Y(8)+\f(10)(4)+\f(14)(6)\Y(4)\\
F_{11} &= \Y(4)\Y(7)-\Y(3)\Y(8)+\f(14)(1)\Y(4)+\f(11)(3)\Y(8)+\f(11)(5)\\
F_{12} &= \Y(4)\Y(8)-X^{2\tau+2}+\f(16)(2)\Y(4)+\f(12)(3)\Y(3)+\f(12)(4)\Y(8)+\f(12)(5)\Y(7)+\f(12)(6)\\
F_{14} &= \Y(7)^2-X^{\tau+1}\Y(8)+\f(14)(1)\Y(7)+\f(14)(2)+\f(14)(4)\Y(4)+\f(14)(5)\Y(3)+\f(14)(6)\Y(8)\\
F_{15} &= \Y(7)\Y(8)-X^{\tau+2}\Y(3)+\f(16)(2)\Y(7)+\f(15)(3)+\f(14)(5)\Y(4)+\f(15)(6)\Y(3)\\
F_{16} &= \Y(8)^2-X^{\tau+2}\Y(4)+\f(16)(1)\Y(3)+\f(16)(2)\Y(8)+\f(12)(3)\Y(7)+\f(16)(4)+\f(15)(6)\Y(4)
\end{aligned}
\end{equation}
involving 25 coefficients $\f(i)(j)$.

To proceed further we note that  the equations \eqref{Reqns} still include 
$X$, which is a function on the monomial curve, in an explicit manner. It turns
out to be more convenient to use a local coordinate at infinity on this curve.
We put 
\[
X=t^{-6},\quad Y_3= t^{-6-3\tau},\quad Y_4=t^{-6-4\tau}, \quad
Y_7= t^{-6-7\tau}, \quad Y_8=t^{-8-3\tau}
\]
in the functions $F_i$ and the syzygies and clear denominators.
We give the variable $t$ the weight $-1$.
We consider the polynomials
 \[
 f_i:=\displaystyle\sum_{r=1}^{12\tau+i}F_i(t^{-6}, t^{-6-3\tau}, t^{-6-4\tau}, 
 t^{-6-7\tau}, t^{-8-3\tau})t^{i+12\tau}\
 \]
 and we write each one as the sum of its partial polynomials 
 \[
 f_i^{(j)}=\displaystyle\sum_{r\equiv j\mod{6}}^{}f_{ir}t^r,\ (j=1,\ldots,6),
 \]
 which are defined by collecting all terms whose exponents are in the 
 same residue class module $6$. In particular, the weight of each
 $f_i^{(j)}$ is now zero. Our previous
 $f_i^{(j)}=\sum_{k=0}^{\rho}f_{i,j+6k}X^{\rho-k}$
 becomes $f_i^{(j)}=\sum_{k=0}^{\rho}f_{i,j+6k}t^{j+6k}$, justifying the use of the 
 same notation. By using the same substitution on the syzygies, like \eqref{syz1},
 the left hand side becomes identically equal to the right hand side, as they give
 the unique expression of the restriction of both sides to the curve.
 The advantage is that the variable $t$ does no longer occur explicitly.
 
 We give the resulting equation from the first of the eight syzygies \eqref{syz}:
\[
f_6-f_7-f_{10}-\f(10)(2)f_8+\f(11)(3)f_7-\f(6)(4)f_{12}+\f(12)(4)f_6-\f(7)(5)f_{11}+\f(8)(6)f_{10} =0
\]
Writing out these equations for all syzygies in terms of the partial polynomials leads to  60
equations, some of which are zero while others coincide.
We further eliminate variables.
Equation $R_1$ in \eqref{syz1} becomes (using the linear equations) $\f(7)(1)=-\f(10)(2)\f(12)(5)-\f(6)(4)\f(12)(3)$.
Similarly we can eliminate all other $\f(i)(j)$ with $i-j\equiv 0 \bmod 6$, which are the
ones in \eqref{fam1pol4} that are not coefficients of an $\Y(i)$.
\begin{equation}
\label{without-y}
\begin{aligned}
\f(6)(6)&=\f(14)(6)+\f(8)(6)+\f(7)(5)\f(14)(1)+\f(6)(4)\f(16)(2)-\f(14)(6)\f(8)(6)\\
\f(7)(1)&=-\f(6)(4)\f(12)(3)-\f(10)(2)\f(12)(5)\\
\f(8)(2)&=-\f(7)(5)\f(12)(3)-\f(11)(3)\f(12)(5)\\
\f(10)(4)&=\f(14)(4)+\f(10)(2)\f(16)(2)-\f(7)(5)\f(14)(5)-\f(14)(4)\f(8)(6)\\
\f(11)(5)&=\f(14)(5)+\f(16)(1)\f(6)(4)+\f(10)(2)\f(12)(3)+\f(14)(1)\f(12)(4)-\f(14)(5)\f(8)(6)\\
\f(12)(6)&=\f(15)(6)+\f(8)(6)-\f(16)(1)\f(7)(5)-\f(11)(3)\f(12)(3)+\f(16)(2)\f(12)(4)-\f(15)(6)\f(8)(6)\\
\f(14)(2)&=\f(10)(2)+\f(14)(4)\f(12)(4)-\f(11)(3)\f(14)(5)-\f(10)(2)\f(15)(6)\\                 
\f(15)(3)&=\f(12)(3)-\f(14)(4)\f(12)(5)-\f(12)(3)\f(14)(6)\\
\f(16)(4)&=\f(14)(1)\f(12)(3)-\f(14)(5)\f(12)(5)
\end{aligned}
\end{equation}
Inserting these values reduces the number of equations to five.
\begin{lem}
The remaining $16$ partial polynomials  $f_{ij}$ satisfy the following five equations identically in $t$:
\begin{equation}
\label{baseeqns11}
\begin{aligned}
\f(16)(1)(1-\f(8)(6))+\f(8)(1)(1-\f(15)(6))+\f(12)(3)\f(12)(4)+\f(16)(2)\f(12)(5)=0\\
\f(8)(1)(1-\f(14)(6))-\f(14)(1)(1-\f(8)(6))+\f(7)(5)\f(16)(2)+\f(6)(4)\f(12)(3)=0\\
\f(11)(3)(1-\f(8)(6))+\f(8)(1)\f(10)(2)-\f(7)(5)\f(12)(4)+\f(6)(4)\f(12)(5)=0\\
\f(6)(4)(1-\f(15)(6))-\f(14)(4)(1-\f(8)(6))-\f(12)(4)(1-\f(14)(6))-\f(10)(2)\f(16)(2)=0\\
\f(7)(5)(1-\f(15)(6))+\f(14)(5)(1-\f(8)(6))-\f(12)(5)(1-\f(14)(6))+\f(10)(2)\f(12)(3)=0.
\end{aligned}
\end{equation}
\end{lem}

The  partial polynomials  $\f(6)(4)$, $\f(7)(5)$ and $\f(8)(6)$ contain $\tau$ variables,
$\f(8)(1)$, $\f(12)(5)$,  $\f(11)(3)$, $\f(10)(2)$, $\f(12)(4)$ and  $\f(14)(6)$  
contain $\tau+1$ variables,
 $\f(12)(3)$, $\f(14)(5)$, $\f(15)(6)$,  $\f(14)(1)$, $\f(14)(4)$,   $\f(16)(2)$  contain $\tau+2$ variables
and $\f(16)(1)$   has $\tau+3$ variables. 
Using the remaining normalizations  $f_{8,1}=f_{12,4}=f_{8,6}=0$
the total number of variables is reduced by three to $16\tau +21$.

To understand the equations  \eqref{baseeqns11} we return to the $\f(i)(j)$ in terms of the variable $X$ and write the
equations somewhat differently.
\begin{equation}
\label{baseeqns12}
\begin{aligned}
(\f(16)(1)+X^2\f(8)(1))(X^\tau-\f(8)(6))&=\f(8)(1)(\f(15)(6)-X^2\f(8)(6))-\f(12)(3)\f(12)(4)-\f(16)(2)\f(12)(5)\\
(\f(14)(1)-X\f(8)(1))(X^\tau-\f(8)(6))&=\f(7)(5)\f(16)(2)-\f(8)(1)(\f(14)(6)-X\f(8)(6))+\f(6)(4)\f(12)(3)\\
\f(11)(3)(X^\tau-\f(8)(6))&=\f(7)(5)\f(12)(4)-\f(8)(1)\f(10)(2)-\f(6)(4)\f(12)(5)\\
(\f(14)(4)+X\f(12)(4)-X^2\f(6)(4))(X^\tau-\f(8)(6))&=\f(12)(4)(\f(14)(6)-X\f(8)(6))-\f(6)(4)(\f(15)(6)-X^2\f(8)(6))-\f(10)(2)\f(16)(2)\\
(\f(14)(5)-X\f(12)(5)+X^2\f(7)(5))(X^\tau-\f(8)(6))&=\f(7)(5)(\f(15)(6)-X^2\f(8)(6))-\f(12)(5)(\f(14)(6)-X\f(8)(6))-\f(10)(2)\f(12)(3).
\end{aligned}
\end{equation} 
All these equations are of the form
\[
L\cdot (X^\tau- \f(8)(6))=R
\]
with $L$ and $R$ polynomials in $X$ satisfying
$\deg_X(R)\leq \deg_X(L)+\tau$. Division with remainder gives
$R=Q(X^\tau- \f(8)(6))+\overline R$, and therefore we can solve $L=Q$ and
find as equations for the base space that the coefficients of $\overline R$
have to be zero. In other words, the condition leading to the equations of the base
space is that the right hand side of the equations \eqref{baseeqns12} is divisible by
$X^\tau -\f(8)(6)$. A similar structure first appeared for the base spaces of rational 
surface singularities of multiplicity four \cite{dJvS}.

 \begin{thm}\label{thm1}
 Let $\N$ be the semigroup generated by $6, 3+6\tau, 4+6\tau, 7+6\tau$ and $8+6\tau$ 
 where $\tau$ is a positive integer. 
 The $5\tau$ equations of the base space  $\mathcal{T}^-$  of the versal deformation
in negative degrees of the monomial curve
 $\C_{\N}$  are given by the condition on the 
 $11\tau+8$ coefficients $f_{ik}$ occurring in the polynomials $\f(i)(j)$,
 that the Pfaffians of the skew-symmetric matrix
\begin{equation}\label{pfaff}
\begin{pmatrix}
0&        \f(16)(2)&         \f(8)(1)&  \f(12)(4)&        \f(6)(4)\\         
-\f(16)(2)&0&                \f(12)(3)& \f(15)(6)-X^2\f(8)(6)&\f(14)(6)-X\f(8)(6)\\
-\f(8)(1)& -\f(12)(3)&        0&        \f(12)(5)&        \f(7)(5)\\         
-\f(12)(4)&-\f(15)(6)+X^2\f(8)(6)&-\f(12)(5)&0&               \f(10)(2)\\        
-\f(6)(4)& -\f(14)(6)+X\f(8)(6)&-\f(7)(5)& -\f(10)(2)&       0                
\end{pmatrix} 
\end{equation}
are divisible by $X^\tau -\f(8)(6)$. In particular, the dimension of 
$T^{1,-}({\k[\N]})$ is $11\tau+8$.
\end{thm}

\begin{proof}
The Pfaffians of the matrix give the same set of equations as the
the right hand side of the equations \eqref{baseeqns12}.  They involve only the
polynomials
$ \f(6)(4)$, $ \f(7)(5)$, $\f(8)(6)$, $ \f(8)(1)$, $ \f(12)(5)$, $ \f(10)(2)$, $\f(12)(4)$, $\f(14)(6)$, $   \f(12)(3)$, $ \f(15)(6)$ and $\f(16)(2)$, all others are eliminated. The divisibility condition
leads to $5\tau$ equations without linear parts. By counting the number of coefficients   and taking the remaining three normalizations
into account we find that the dimension of 
$T^{1,-}({\k[\N]})$ is $11\tau+8$.
\end{proof}

\begin{cor}
The isomorphism classes of the pointed complete 
 integral Gorenstein curves with Weierstrass semigroup $\N$ correspond bijectively to 
 the orbits of the $\mathbb{G}_m$-action on $ \mathcal{T}^-$ described in the theorem above.
\end{cor}

\begin{thm}\label{negsmoot1}
The monomial curve $C_\N$ is negatively smoothable. Therefore the 
moduli space $\M$ is non-empty.
\end{thm}

\begin{proof}
The equations \eqref{pfaff} can be solved by setting all variables equal to zero
except for those involved in $\f(8)(6)$, $\f(14)(6)$ and $\f(15)(6)$. We take only the
variables of highest weight non-zero. Using \eqref{without-y} we find the following
generators of the ideal:
\begin{equation}\label{smoothcurve}
\begin{array}{lll}
  Y_3^2-(X^{\tau+1}-b)(X^\tau-a),&  Y_3Y_4-(X^\tau-a)Y_7,&  Y_4^2-(X^\tau-a)Y_8,\\
  Y_3Y_7-(X^{\tau+1}-b)Y_4,& Y_4Y_7-Y_3Y_8,& Y_4Y_8-(X^{\tau+2}-c)(X^\tau-a),\\
  Y_7^2-(X^{\tau+1}-b)Y_8,& Y_7Y_8-(X^{\tau+2}-c)Y_3,& Y_8^2-(X^{\tau+2}-c)Y_4,
\end{array}
\end{equation}
where $a$, $b$ and $c$ are non-zero constants such that the polynomials 
$X^{\tau}-a$, $X^{\tau+1}-b$ and $X^{\tau+2}-c$ have pairwise no common roots.
To show that this curve is smooth we compute the Jacobi-matrix of the above system
of equations. First consider the submatrix
\[
\begin{pmatrix}
-\frac{\partial (X^{\tau+1}-b)(X^\tau-a)}{\partial X} & 2 Y_3 & 0 & 0 & 0 \\
-\frac{\partial (X^{\tau+2}-c)(X^\tau-a)}{\partial X} & 0 &  Y_8 & 0 & Y_4 \\
-(\tau+2)X^{\tau+1}Y_3 & -(X^{\tau+2}-c) & 0  & Y_8 & Y_7\\
-(\tau+2)X^{\tau+1}Y_4 & 0 &  -(X^{\tau+2}-c) & 0& 2Y_8
\end{pmatrix}
\]
The lower right $3\times3$ subdeterminant is equal to
$Y_8(2Y_8^2+(X^{\tau+2}-c)Y_4)$, which is congruent to $3Y_8^3$ modulo
the last equation. A necessary condition for a singular point is therefore
$Y_3Y_8=0$. With the other equations it then follows that $Y_7Y_8=0$
and therefore $Y_7=(X^{\tau+1}-b)Y_8=0$. From the submatrix again
we also get that $X^\tau(X^{\tau}-a)Y_8=0$.
Because $\gcd(X^\tau(X^{\tau}-a),X^{\tau+1}-b)=1$ it follows that $Y_8=0$
and then $Y_4=0$. The Jacobi-matrix simplifies, and we can conclude
that $X^{\tau+2}-c=Y_3=0$. But then also
$(X^{\tau+1}-b)(X^\tau-a)=0$, contradicting the choice of $a$, $b$ and $c$.
Therefore this curve provides a negative smoothing. By Pinkham's theorem
$\M$ is then non-empty. 
\end{proof}

\comment
 \f(6)(4), \f(7)(5), \f(8)(6), \f(8)(1), \f(12)(5), \f(10)(2), \f(12)(4),  \f(14)(6),   \f(12)(3), \f(15)(6) ,  \f(16)(2)  

 $f_{8,1}=f_{12,4}=f_{15,6}=0$.

\endcomment

Although we get explicit equations for the base space, it is difficult to understand 
the structure of this space, or even to find its dimension.  We note that the 
dimension is equal to the dimension of the tangent cone. Among the equations 
of the tangent cone we have the quadratic part of our $5\tau$ equations.
They define  an affine quadratic 
cone
$\mathcal{Q}\subset\mathbb{A}^{11\tau+8}$, 
that contains the tangent cone. In this way we get an upper bound for the
dimension; this is the method
presented in \cite[Section 3]{c2013}.

Note  that division by $X^\tau -\f(8)(6)$ introduces higher order monomials
involving the coefficients of $\f(8)(6)$. So to obtain the quadratic part
of the equations we have to divide by $X^\tau$.
We introduce the $\tau$-dimensional Artinian algebra $k[X]/(X^{\tau})$.
\begin{thm}\label{teo12}
The quadratic quasi-cone $\mathcal{Q}$ is isomorphic to the direct product
\[
\mathcal{Q}=M\times N,
\]
where $M$ is the $(\tau+8)$-dimensional weighted affine space of weights $2, 2, 3, 5, 6, 6, 8, 9, 12$ 
and $6i, i=1,\ldots,\tau-1$, and $N$ is the quadratic quasi-cone consisting of vectors
\[
(\omega_1,\ldots,\omega_{10})=\left(\displaystyle\sum_{j=0}^{\tau-1}\omega_{1j}X^j,\ldots,
\displaystyle\sum_{j=0}^{\tau-1}\omega_{10,j}X^j\right),
\]
such that the Pfaffians of the matrix
\[
\begin{pmatrix}
0&        \omega_1&         \omega_2&  \omega_3&        \omega_4\\         
-\omega_1&0&               \omega_5& \omega_6&\omega_7\\
-\omega_2& -\omega_5&        0&       \omega_8&    \omega_9\\         
-\omega_3&-\omega_6&-\omega_8&0&          \omega_{10}\\        
-\omega_4& -\omega_7&-\omega_9& -\omega_{10}&       0                
\end{pmatrix} 
\]
are zero in the artinian algebra $k[X]/(X^{\tau})$.
\end{thm}
\begin{proof}
We first put $X^\tau=0$ in the matrix \eqref{pfaff}. Then we have a 
skew-symmetric matrix whose ten entries above the diagonal are polynomials
of degree $\tau-1$ with in total $10\tau$ linearly independent coefficients, so a 
generic matrix of this type. By an obvious substitution we obtain the matrix 
in the statement.
\end{proof}

By the computation in  the proof of \cite[Cor. 4.5]{c2013} the dimension of
$N$ is $7\tau$. Using the lower bound from Theorem  \ref{Stevens} we obtain 
\begin{cor}\label{cor13}
We have $\dim\mathcal{Q}=8\tau+8$. The moduli space $\M$ 
is of pure dimension $8\tau+7$.
\end{cor}

 \subsection{A second family}

 We apply the same method as above for the following particular family
 of symmetric semigroups. For each $\tau\geq 1$, let
 
 \begin{align*}
\N&=\langle 6, 1+6\tau, 2+6\tau, 3+6\tau, 4+6\tau\rangle\\
&=\mathbb{N}\sqcup\bigsqcup_{j\in\{1,2,3,4\}}^{}(j+6\tau+6\mathbb{N})
\sqcup(5+12\tau+6\mathbb{N}),
\end{align*}
be a symmetric semigroup of genus $g=6\tau$ generated minimally by five elements. 

\begin{rmk}
If $\tau=1$, the  ideal of the canonical monomial
curve $\C_{\N}\subset \proj^5$ can not be generate by quadratic forms only, see
the recent preprint by Contiero and Fontes  \cite{CF}.
In this case there are two natural compactifications of the monomial
curve in $A^5$: in $\proj^5$ the equations of the next lemma define the ideal
of the canonical monomial curve; it is a trigonal curve whose ideal
can be generated by 6 quadratic and 3 cubic equations.  For
Pinkham's construction we compactify the affine curve in a weighted
projective space with weights $(1,6,7,8,9,10)$, and the results below are 
also valid for $\tau=1$.
\end{rmk}

 Since the methods are the same
as in the preceding subsection, we do not give  proofs.
We introduce  variables $X, Y_1, Y_2, Y_3, Y_4$ of weight
$6, 1+6\tau,  2+6\tau, 3+6\tau, 4+6\tau$. Similar to Lemma \ref{lem11}
we have

\begin{lem}\label{lem21} The ideal of the affine monomial curve
  \[
\C_{\N}=\{(t^6, t^{1+6\tau}, t^{2+6\tau}, t^{3+6\tau}, t^{4+6\tau}):t\in\k\}
\]
is generated by the  forms $F_i^{(0)}$ and $G_i^{(0)}$
\begin{align*}
F_{2}^{(0)}&= Y_1^2-X^{\tau}Y_2,&  
F_{3}^{(0)}&= Y_1Y_2-X^{\tau}Y_3,&  
F_{4}^{(0)}&=Y_1Y_3-X^{\tau}Y_4,\\
G_{4}^{(0)}&= Y_2^2-X^{\tau}Y_4,& 
F_{5}^{(0)}&=Y_1Y_4-Y_2Y_3,& 
F_{6}^{(0)}&=Y_2Y_4-X^{2\tau+1},\\
G_{6}^{(0)}&= Y_3^2-X^{2\tau+1},& 
F_{7}^{(0)}&=Y_3Y_4-X^{\tau+1}Y_1,&
F_{8}^{(0)}&=Y_4^2-X^{\tau+1}Y_2.
\end{align*}
   \end{lem}


We unfold the  forms $F_i^{(0)}$ and $G_i^{(0)}$ with terms of lower degree:

\begin{equation}
\label{fam2pol3}
\begin{aligned}
F_2 &=  \Y(1)^2-X^\tau  \Y(2)+ \f(2)(1)  \Y(1)+ \f(2)(2)+ \f(2)(4)  \Y(4)+ \f(2)(5)  \Y(3)+ \f(2)(6)  \Y(2)\\
F_3 &=  \Y(1)  \Y(2)-X^\tau  \Y(3)+ \f(3)(1)  \Y(2)+ \f(3)(2)  \Y(1)+ \f(3)(3)+ \f(3)(5)  \Y(4)+ \f(3)(6)  \Y(3)\\
F_4 &=  \Y(1)  \Y(3)-X^\tau  \Y(4)+ \f(4)(1)  \Y(3)+ \f(4)(2)  \Y(2)+ \f(4)(3)  \Y(1)+ \f(4)(4)+ \f(4)(6)  \Y(4)\\
G_4 &=  \Y(2)^2-X^\tau  \Y(4)+ \g(4)(1)  \Y(3)+ \g(4)(2)  \Y(2)+ \g(4)(3)  \Y(1)+ \g(4)(4)+ \g(4)(6)  \Y(4)\\
F_5 &=  \Y(1)  \Y(4)- \Y(2)  \Y(3)+ \f(5)(1)  \Y(4)+ \f(5)(2)  \Y(3)+ \f(5)(3)  \Y(2)+ \f(5)(4)  \Y(1)+ \f(5)(5)\\
F_6 &=  \Y(2)  \Y(4)-X^{2\tau+1}+ \f(6)(1)  \Y(2)  \Y(3)+ \f(6)(2)  \Y(4)+ \f(6)(3)  \Y(3)+ \f(6)(4)  \Y(2)+ \f(6)(5)  \Y(1)+ \f(6)(6)\\
G_6 &=  \Y(3)^2-X^{2\tau+1}+ \g(6)(1)  \Y(2)  \Y(3)+ \g(6)(2)  \Y(4)+ \g(6)(3)  \Y(3)+ \g(6)(4)  \Y(2)+ \g(6)(5)  \Y(1)+ \g(6)(6)\\
F_7 &=  \Y(3)  \Y(4)-X^{\tau+1}  \Y(1)+ \f(7)(1)+ \f(7)(2)   \Y(2)  \Y(3)+ \f(7)(3)  \Y(4)+ \f(7)(4)  \Y(3)+ \f(7)(5)  \Y(2)+ \f(7)(6)  \Y(1)\\
F_8 &=  \Y(4)^2-X^{\tau+1}  \Y(2)+ \f(8)(1)  \Y(1)+ \f(8)(2)+ \f(8)(3)  \Y(2)  \Y(3)+ \f(8)(4)  \Y(4)+ \f(8)(5)  \Y(3)+ \f(8)(6)  \Y(2)
\end{aligned}
\end{equation}
Here $ \f(i)(j)$ or $ \g(i)(j)$  is a polynomial in $X$ with deformation variables as
coefficients, of  degree $2\tau$ in $X$ if $i=j$,
$2\tau+1$ if $i=j+6$,  degree 0 if $i=j+5$ and otherwise 
of degree $\tau+\lfloor \frac{i-j}6\rfloor$ in $X$.

We normalize the coefficients
\[
\f(6)(1)= \g(6)(1)= \f(7)(2)= \f(8)(3)=0,\quad  \f(4)(1)= \f(5)(2)= \f(7)(3)= \f(5)(4)=0.
\]
We have three normalizations left, which we leave for later on.
These are the following: we can make $f_{2,6}$, $f_{3,6}$ or $f_{8,6}$
equal to zero, we can make $f_{2,1}$, $g_{4,1}$ or $f_{8,1}$
equal to zero and we can make $f_{4,2}$ or $f_{5,2}$ 
equal to zero (in fact also other coefficients, but they will turn out to be equal
to one of these).

With these normalizations the equations become similar to the normalized
equations \eqref{fam1pol3}. In fact, if we replace $Y_1$ by $Y_7$ and 
$Y_2$ by $Y_8$, and rename the equations accordingly:
\[
\begin{array}{llllllllll}
\text{equations  \eqref{fam2pol3}:}&
F_2 & F_3 & F_4 & G_4 & F_5 & F_6 & G_6 & F_7&  F_8\\
\text{equations  \eqref{fam1pol3}:}&
F_{14} & F_{15} & F_{10} & F_{16} & F_{11} & F_{12} & F_6 & F_7&  F_8
\end{array}
\]
then the equations become identical as written  but the symbols $\f(i)(j)$
have a slightly different meaning, and the powers of $X$ are somewhat different.
The equations for the partial polynomials in $t$ become the same. 
Therefore we find the equations for the base space by
renaming the partial polynomials in \eqref{baseeqns11}, and
we find
\begin{equation}
\label{baseeqns21}
\begin{aligned}
\g(4)(1)(1-\f(8)(6))+\f(8)(1)(1-\f(3)(6))+\f(6)(3)\f(6)(4)+\g(4)(2)\f(6)(5)&=0\\
\f(8)(1)(1-\f(2)(6))-\f(2)(1)(1-\f(8)(6))+\f(7)(5)\g(4)(2)+\g(6)(4)\f(6)(3)&=0\\
\f(5)(3)(1-\f(8)(6))+\f(8)(1)\f(4)(2)-\f(7)(5)\f(6)(4)+\g(6)(4)\f(6)(5)&=0\\
\g(6)(4)(1-\f(3)(6))-\f(2)(4)(1-\f(8)(6))-\f(6)(4)(1-\f(2)(6))-\f(4)(2)\g(4)(2)&=0\\
\f(7)(5)(1-\f(3)(6))+\f(2)(5)(1-\f(8)(6))-\f(6)(5)(1-\f(2)(6))+\f(4)(2)\f(6)(3)&=0.
\end{aligned}
\end{equation}
These equations, written with the variable $X$, are less suited as the term we have to divide with is
$X^{\tau+1}-\f(8)(6)$. We replace the system of equations with
an equivalent one, where the second and third equation are the following:
\begin{align*}
\g(4)(1)(1-\f(2)(6))+\f(2)(1)(1-\f(3)(6))-\f(6)(3)\f(2)(4)+\g(4)(2)\f(2)(5)&=0\\
\f(5)(3)(1-\f(3)(6))-\g(4)(1)\f(4)(2)+\f(6)(4)\f(2)(5)+\f(2)(4)\f(6)(5)&=0
\end{align*}
We now can use division with $X^{\tau}-\f(3)(6)$. We write
\begin{equation}
\label{baseeqns22}
\begin{aligned}
(\f(8)(1)+\g(4)(1))(X^\tau-\f(3)(6))&=\g(4)(1)(\f(8)(6)-X\f(3)(6))-\f(6)(3)\f(6)(4)-\g(4)(2)\f(6)(5)\\
(\g(4)(1)+\f(2)(1))(X^\tau-\f(3)(6))&=\g(4)(1)(\f(2)(6)-\f(3)(6))+\f(6)(3)\f(2)(4)-\g(4)(2)\f(2)(5)\\
\f(5)(3)(X^\tau-\f(3)(6))&=\g(4)(1)\f(4)(2)-\f(6)(4)\f(2)(5)-\f(2)(4)\f(6)(5)\\
(\g(6)(4)-X\f(2)(4)-\f(6)(4))(X^\tau-\f(3)(6))&=-\f(2)(4)(\f(8)(6)-X\f(3)(6))-\f(6)(4)(\f(2)(6)-\f(3)(6))+\f(4)(2)\g(4)(2)\\
(\f(6)(5)-\f(7)(5)-\f(2)(5))(X^\tau-\f(3)(6))&=-\f(2)(5)(\f(8)(6)-X\f(3)(6))+\f(6)(5)(\f(2)(6)-\f(3)(6))+\f(4)(2)\f(6)(3)
\end{aligned}
\end{equation} 
We write the right hand side as Pfaffians of the following skew-symmetric matrix
\begin{equation}\label{pfaff2}
\begin{pmatrix}
0&        \g(4)(2)&         \g(4)(1)&  \f(6)(4)&        -\f(2)(4)\\         
-\g(4)(2)&0&                \f(6)(3)& \f(8)(6)-X\f(3)(6)&\f(2)(6)-\f(3)(6)\\
-\g(4)(1)& -\f(6)(3)&        0&        \f(6)(5)&        \f(2)(5)\\         
-\f(6)(4)&-\f(8)(6)+X\f(3)(6)&-\f(6)(5)&0&               \f(4)(2)\\        
\f(2)(4)& -\f(2)(6)+\f(3)(6)&-\f(2)(5)& -\f(4)(2)&       0                
\end{pmatrix} 
\end{equation}

We conclude that the base space has the same structure as for the 
first family.

\begin{thm}\label{thm2}
 Let $\N$ be the semigroup generated by $6, 1+6\tau, 2+6\tau, 3+6\tau$ and $4+6\tau$ 
 where $\tau$ is a positive integer. 
 The $5\tau$ equations of the base space  $ \mathcal{T}^-$  of the versal deformation of the monomial curve
 $\C_{\N}$ in negative degrees are given by the condition on the 
 $11\tau+4$ coefficients  occurring in the matrix \eqref{pfaff2}
 that the 
 Pfaffians of this  matrix
are divisible by $X^\tau -\f(3)(6)$. In particular, the dimension of 
$T^{1,-}({\k[\N]})$ is $11\tau+4$.
\end{thm}

As for the curves in the first family we can show that the monomial curve 
is negatively smoothable.

\begin{cor}\label{boundfam2}
The  dimension of the affine cone $\mathcal{Q}_\N$
given by the quadratic part of the equations is $8\tau + 4$
and $\M$ has pure dimension $8\tau+3$.
\end{cor}

\begin{rmk}
Using the upper bound obtained by Contiero and Stoehr in \cite[Cor. 4.5]{c2013}
for the symmetric semigroup generated minimally by $6, 2+6\tau, 3+6\tau, 4+6\tau$ and $5+6\tau$ for $\tau\geq 1$ we find that also in this case $\M$ has pure dimension
$2g-1-\dim\mathrm{T}^{1,+}=8\tau+7$, if non-empty.
\end{rmk}

\bigskip

\parbox[t]{3in}{{\rm Andr\'e Contiero}\\
{\tt \href{mailto:contiero@ufmg.br}{contiero@ufmg.br}}\\
{\it Universidade Federal de Minas Gerais}\\
{\it Belo Horizonte, MG, Brazil}} \hspace{0.9cm}
\parbox[t]{3in}{{\rm Aislan L. Fontes}\\
{\tt \href{mailto:ailfontes@hotmail.com }{aislan@ufs.br}}\\
{\it Universidade Federal de Sergipe}\\
{\it Itabaiana, SE, Brazil}} 

\bigskip

\parbox[t]{3in}{{\rm Jan Stevens}\\
{\tt \href{mailto:stevens@chalmers.se}{stevens@chalmers.se}}\\
{\it Department of Mathematical Sciences,\\ Chalmers University of
Technology and University of Gothenburg.\\
SE 412 96 Gothenburg, Sweden}} \hspace{0.9cm}
\parbox[t]{3in}{{\rm Jhon Quipse Vargas}\\
{\tt \href{jhon.quispev@gmail.com }{jhon.quispev@gmail.com }}\\
{\it Universidade Federal de Goias}\\
{\it Catal\~ao, GO, Brazil}} 
 
 \end{document}